\documentclass[12pt]{article}
\usepackage{tikz}
\usepackage{microtype}
\usepackage{geometry}                
\usepackage{amssymb,amsmath,amsthm}
\newtheorem*{theorem}{Theorem}
\newtheorem{lemma}{Lemma}
\theoremstyle{remark}

\newtheorem*{definitions}{Definitions}
\usepackage{color}
\usepackage{hyperref}
\usepackage{wrapfig}
\usepackage{graphicx}
\usepackage{url}

\usepackage{ifxetex}
\ifxetex
\usepackage{fontspec,xltxtra,xunicode}
\usepackage{unicode-math}
\defaultfontfeatures{Mapping=tex-text}
\newfontfamily\cyrillicfont{STIX Two Text}
\setromanfont[Mapping=tex-text]{STIX Two Text}
\usepackage{polyglossia}
\setotherlanguage{russian}
\setotherlanguage{german}
\setdefaultlanguage{english}
\else
\usepackage{mathptmx}
\fi


\begin{document}
\title{Axiomatic approach to the theory of algorithms and relativized computability\footnote{Vestnik Moskovskogo Universiteta.  (Bulletin of the Moscow State University). Ser. 1: Mathematics, Mechanics. 1980, no.2, p.~27--29. Translated by the author (2018)}}
\author{Alexander Shen}
\date{}

\maketitle

It is well known that many theorems in recursion theory can be ``relativized''. This means that they remain true if partial recursive functions are replaced by functions that are partial recursive relative to some fixed oracle set. Uspensky in~\cite{uspensky1974} formulates three ``axioms'' called ``axiom of computation records'', ``axiom of programs'' and ``arithmeticity axiom''. Then, using these axioms (more precisely, two first ones) he proves basic results of the recursion theory. These two axioms are true also for the class of functions that are partial recursive relative to some fixed oracle set. Also this class is closed under substitution, primitive recursion and minimization ($\mu$-operator); these (intuitively obvious) closure properties are also used in the proofs.  This observation made by Uspensky explains why many theorems of recursion theory can be relativized. It turns out that the reverse statement is also true: all relativizable results follow from the first two axioms and closure properties. Indeed, \emph{every class of partial functions that is closed under substitution, primitive recursion and minimization that satisfies the first two axioms is the class of functions that are partial recursive relative to some oracle set $A$}. This is the main result of the present article.
\medskip

Let $K$ be a class of partial functions with natural arguments and values. It may contain functions of different arity. Consider the following requirements for the class~$K$:

\begin{enumerate}
\item\label{closure} (Closure properties) The class $K$ contains all partial recursive functions and is closed under substitution, primitive recursion and $\mu$-operator.
\item\label{records} (Computation records) For every unary function $\in K$ there exists a set $M$ of natural numbers and functions $\alpha,\omega\in K$ whose domains contain $M$ such that
\begin{enumerate}
\item The indicator function of the set $M$ that is equal to $1$ on $M$ and is equal to $0$ outside $M$, belongs to $K$;
\item the value of $f$ on some $x$ is defined and equal to some $y$ if and only if there exists some $m\in M$ such that $\alpha(m)=x$ and $\omega(m)=y$. 

(Using the terminology from~\cite{uspensky1974}, one may say that $M$ is the set of all computation records for the function $f$ and all  possible inputs; the functions $\alpha$ and $\omega$ are defined on all computation records and extract the input and output respectively.)
\end{enumerate}
\item\label{programs} (Programs axiom) There exist a binary function $F\in K$ that is universal for the unary functions in $K$. This means that for every unary function $f\in K$ there exists some $n$ such that the function $F(n,\cdot)$ coincides with $f$. 
\end{enumerate}

\begin{theorem}
Let $K$ be a class that satisfies the conditions~\ref{closure}--\ref{programs}. Then there exists a set $A$ of natural numbers such that $K$ is the class of all functions whose graph is enumeration reducible to the graph of $f$.
\end{theorem}
\begin{proof}[Proof of the theorem]
Let $F$ be the binary universal function for $K$ that exists according to the condition~\ref{programs}.
\begin{lemma}\label{lemma1}
An arbitrary partial function $f$ belongs to $K$ if and only if $f$ is partial recursive relative to $F$.
\end{lemma}
\begin{proof}[Proof of the lemma~\ref{lemma1}]

(`Only if' part) If $f$ is a unary function, we use the universality condition. The general case can be reduced to the unary case using the recursive numbering of tuples.

(`If' part) For that we need first to prove some results mentioned above, i.e., to develop some fragment of the recursion theory based on axioms \ref{closure}--\ref{programs}.
\begin{definitions}
A. The set $X$ [of natural numbers] is called \emph{$K$-enumerable} if $X$ is empty or $X$ is the range of a total unary function from $K$. (One may omit the word `unary' without changing the class of $K$-enumerable sets.)

B. A set $X$ is \emph{$K$-decidable} if its indicator function belongs to $K$.
\end{definitions}
Using these definitions, one may easily prove the following statements:

\begin{description}
\item[1.1] Every recursively enumerable set is $K$-enumerable.
\item[1.2] Every decidable set is $K$-decidable.
\item[1.3] Every $K$-decidable set is $K$-enumerable.
\item[1.4] A set $X$ is $K$-decidable if and only if its complement $\overline X$ is $K$-decidable.
\item[1.5] A set $X$ is $K$-decidable if and only if both $X$ and its complement $\overline X$ are $K$-enumerable.
\item[1.6] The union, intersection and set difference of two $K$-decidable sets are $K$-decidable.
\item[1.7] The union and intersection of two $K$-enumerable sets are $K$-enumerable.
\item[1.8] The projection of a planar $K$-enumerable set  (i.e., a set of pairs whose codes in the standard recursive encoding of pairs form a $K$-enumerable set) is $K$-enumerable..
\item[1.9] Every $K$-enumerable set of natural numbers is a projection of some $K$-decidable planar set.
\item[1.10] If a partial function has a $K$-enumerable graph, then this function belongs to $K$.
\item[1.11] If a set  $X$ is $K$-enumerable (resp. $K$-decidable), then the set \[\{ x\mid \text{$x$th tuple has all elements on $X$}\}\] is $K$-enumerable (resp. $K$-decidable).
\end{description}
The proofs of 1.1--1.11 use only the closure property; the following results also use the computation records axiom.
\begin{description}
\item[2.1] The domain of every function from $K$ is $K$-enumerable.
\item[2.2] The range of every function from $K$ is $K$-enumerable.
\item[2.3] The graph of every function from $K$ is $K$-enumerable.
\item[2.4] The images and preimages of $K$-enumerable sets under functions from $K$ are $K$-enu\-merable.
\end{description}
Replacing recursive functions in the definition of enumeration reducibility (see~\cite[section~9.7, p.~146]{rogers1967} by functions from $K$, we get the definition of \emph{$K$-enumeration reducibility}.
\begin{description}
\item[2.5] The set that is $K$-enumeration reducible to a $K$-enumerable set is $K$-enumerable.
\end{description}
Now we return to the proof of the `if' part.  Since $F$ belongs to $K$, its graph is $K$-enumerable. Since the graph of $f$ is enumeration reducible to the graph of $F$, then the graph of $f$ is $K$-enumeration reducible to the graph of $F$. Then \textbf{2.5} says that the graph of $f$ is $K$-enumerable, and therefore $f\in K$. Lemma~\ref{lemma1} is proven.
\end{proof}
\begin{lemma}\label{lemma2}
Assume that for some set $M$ the class of all functions whose graphs are enumeration reducible to $M$ satisfies the conditions~\ref{closure}--\ref{programs}. Then there exists a set $S$ such that $M$ is enumeration equivalent \textup(=reducible in both directions\textup) to the set $S\oplus \overline S$. 
\end{lemma}
(Recall that by definition $A\oplus B$ is $\{2n\mid n\in A\}\cup\{2n+1\mid n\in B\}$.)
\begin{proof}[Proof of Lemma~\ref{lemma2}]
Let $L$ be the class mentioned in the statement of Lemma~\ref{lemma2}. The set $M$ is $L$-enumerable. Now we use that $L$ satisfies the conditions~\ref{closure}--\ref{programs} and note that \textbf{1.9} guarantees that $M$ is a projection of some $L$-deducible set $S$. The statements \textbf{1.3} and \textbf{1.4} guarantee that both $S$ and its complement $\overline S$ are $L$-enumerable. But $L$-enumerability is equivalent to the enumeration reducibility to $M$. Therefore $S$ is enumeration reducible to $M$, the complement of $S$ is enumeration reducible to $M$, and $M$ is enumeration reducible to $S$ (the projection of every set is enumeration reducible to the set itself).
\end{proof}

Now the statement of the theorem easily follows from the two lemmas. Here we use the following observation: the enumeration reducibility to the set $S\oplus \overline S$ is equivalent to recursive enumerability relative to $S$ (see \cite[p.~133]{rogers1967} for the definition). Theorem is proven.
\end{proof}

All three conditions used in the statement of the theorem, are essential. To see that this is the case, we may consider the following three examples of classes that satisfy only two of them and do not have the property mentioned in the theorem.
\begin{enumerate}
\item The class $K$ consists of a identity unary function, constant unary functions and all binary, ternary, etc. functions. It satisfies all the conditions except for the first one (closure).
\item The class $K$ consists of all functions that can be obtained by recursive operators \cite[\S9.8, p.~148]{rogers1967} from the function $\psi$ defined as in~\cite[Theorem XVIII, p.~280]{rogers1967}. Then all the conditions except for the second one (computation records) are satisfied.
\item The class $K$ consists of all partial functions. Then all the conditions except for the last one (programs axiom) are satisfied.
\end{enumerate}

\textbf{Acknowledgment}. The author thanks Vladimir Andreevich Uspensky who asked the question settled in this paper and for the research guidance.

\bigskip
\hbox to \textwidth{%
\hbox{}
\hfill
\hbox{\small Received May 10, 1978}
}
\par

\bigskip
\noindent
[Added when translating] One can also prove easily that axioms \ref{closure} and \ref{programs} alone describe the classes that are partial recursive closures of arbitrary functions. This means that

 (1) if $f$ is some function, that the minimal class of functions that is closed under substitution, recursion, minimization and contains all partial recursive functions and $f$, satisfies axioms \ref{closure} and \ref{programs};
 
(2) all classes that satisfy these two axioms can be obtained in this way for some $f$.
\end{document}